\documentclass[reqno,11pt]{amsart}

\usepackage{times,amsmath,cancel,stmaryrd,graphicx}
\usepackage{amsfonts,enumitem,amssymb,color,mathrsfs}
\usepackage[colorlinks=true]{hyperref}
\usepackage{cleveref,enumerate}
\usepackage{lmodern}
\newtheorem{thm}{Theorem}[section]
\newtheorem{lem}[thm]{Lemma}

\newtheorem{prop}[thm]{Proposition}

\newtheorem{rem}[thm]{Remark}

\numberwithin{equation}{section}

\newcommand{\R}{\mathbb{R}}

\newcommand{\ve}{\varepsilon}
\newcommand{\rd}{\mathrm{d}}

\newcommand{\dhr}{\mathrel{\lhook\joinrel\relbar\kern-.8ex\joinrel\lhook\joinrel\rightarrow}}

\begin{document}

\title[A one-dimensional Biofilm model]{Analysis of a one-dimensional biofilm model}


%
\author{Patrick Guidotti}
\address{University of California, Irvine\\
Department of Mathematics\\
340 Rowland Hall\\
Irvine, CA 92697-3875\\
USA}
\email{pguidott@uci.edu}

\author{Christoph Walker}
\address{Leibniz Universit\"at Hannover\\
Institut f\"ur Angewandte Mathematik\\
Welfengarten 1\\
30167 Hannover\\
Germany}
\email{walker@ifam.uni-hannover.de}
\date{\today}

\begin{abstract}
In this paper a reduced one-dimensional moving boundary model is
studied that describes the evolution of a biofilm driven by the
presence of a reaction limiting substrate. Global well-posedness is
established for the resulting parabolic free boundary value problem in
strong form in Sobolev spaces and for a quasi-stationary approximation
in spaces of classical regularity. The general existence results are
complemented by results about the qualitative properties of solutions
including the existence, in general, and, additionally, the uniqueness
and stability of non-trivial equilibria, in a special
case. 

\end{abstract}
%
\maketitle
\section{Introduction}
Microbial systems are the focus of an increasing body of research due
to their prevalence in many important fields of application.  Microbes
tend to live in localized communities called biofilms. Most of the
research dedicated to biofilms is confined to the biological realm or
to the specific area of application. The systematic development of
mathematical models and their analysis is still quite
limited. Inspired by the review article \cite{KD10}, we consider a
one-dimensional toy model for the evolution of a single biofilm in a
single reaction limiting susbtrate (e.g. oxygen) introduced in
that article. Even in its simplicity the model takes the form of a
non-standard moving boundary problem with interesting features. The
goal of this paper is to provide a rigorous analysis of the system
that include well-posedness as well as qualitative properties such as
global existence and long-time behavior. The model tracks the
substrate, via its concentration $c(t,z)$ everywhere in the upper
half-space (or half-line since we shall assume homogeneity in the
other variables). Its concentration can be assumed constant in the
bulk fluid region far enough from the biofilm, where a well mixed
regime is typically observed. The biofilm is simply modeled by the
height (depth) $h(t)\geq 0$ to which it has grown. Within the biofilm and in a
diffusive boundary layer on top of it, the substrate can freely
diffuse with (slightly) different diffusivities. It is assumed, based
on observations, that the diffusive boundary layer have a given thickness
$L>0$. At the bottom a wall confines the system and a no-flux
condition is enforced. At the interface $z=h(t)$ continuity of the
substrate concentration and flux is imposed along with a dynamic
boundary condition for the evolution of the biofilm thickness. The
latter captures the modeling assumption that the speed of the
interface coincides with the one generated by the pressure change due
to the biofilm growth or decay. More specifically, let $L,c_*>0$,
$\kappa,\kappa_L>0$, $\varepsilon_L,\varepsilon>0$, and consider the
system
\begin{subequations}\label{Bx}
\begin{align}
\varepsilon_L  c_t&=\kappa_Lc_{zz}, &&h(t)<z<h(t)+L,\quad t>0,\label{BB1x}\\
\varepsilon c_t&=\kappa c_{zz}-r(c),  &&0<z<h(t),\\
c(t,h(t)+L)&=c_*,\quad c_z(t,0)=0, &&t>0,\label{BB3x}\\
\llbracket c\rrbracket&=\llbracket \varkappa c_z\rrbracket=0, &&z=h(t),\label{BB4x}\\
c(0,z)&=c_0(z),\: &&0<z<h_0+L, \label{BB5x}\\
h_t&=\int_0^h g\big(r(c(t,z))\big)\,\rd z,&& t>0, \label{BB6x}\\
h(0)&=h_0>0 \label{BB7x},
\end{align}
\end{subequations}
for the unknown biofilm height $h=h(t)$ and substrate $c=c(t,z)$. Here
we have set
$$\varkappa(z)=\begin{cases}
  \kappa_L,&z\in \bigl( h(t),h(t)+L\bigr)\\
  \kappa,&z\in \bigl(0,h(t)\bigr),
\end{cases}
$$
for the substrate diffusivities in the diffusive layer
$[h(t)<z<h(t)+L]$ and in the biofilm $[0<z<h(t)]$, respectively. The
function $r$ models the substrate consumption that drives the biofilm
growth, itself captured by a function $g$ of $r$. A common choice for
the function $g$ is given by $g(r)=\alpha(r-b)$, where $\alpha>0$
represents a yield coefficient and $b$ the basic substinence level of
the biofilm. Minimal assumptions on the function $r$ are that
$r(0)=0$, modeling the fact that no substrate is consumed if none is
present, and that $r(c)>0$ for $c>0$ indicating that available
substrate is always used. The function $r$ is also typically made to reflect a
saturation at increasing levels of the substrate concentration. We
shall not make use of this additional structure in this analysis but
we will assume that $r'(c)>0$ for $c>0$, which does not contradict the
saturation assumption and simply encodes the fact the consumption
increases with substrate concentration. The assumptions about the
substrate beyond the top diffusive layer and the bottom wall are
reflected in the boundary conditions  \eqref{BB3x}, while continuity
across the interface is encoded in the conditions \eqref{BB4x}.
The kynetic condition \eqref{BB6x} for $h$ is derived from an
incompressibility assumption for the fluid (mostly 
water) and Darcy's law connecting velocity and pressure. Full details
can be found in \cite{KD10} or the earlier \cite{KD02}. The remaining
equations \eqref{BB5x} and \eqref{BB7x} fix an arbitrary initial
configuration. 
If diffusion in the diffusive biofilm layer is assumed to take place
at a much faster time scale than biofilm growth, it can be assumed
that $\varepsilon_L=0$ in \eqref{BB1x}. In that case, $c$ is affine on
$(h,h+L)$ and it follows from~\eqref{BB3x} and $\llbracket c\rrbracket=0$ on $z=h(t)$ that
$c$ is determined on $(h,h+L)$ by $c\bigl(t,h(t)\bigr)$ through
\begin{equation}\label{e0x}
c(t,z)=\frac{c_*-c\bigl(t,h(t)\bigr)}{L}\big[z-h(t)\big]+c\bigl(t,h(t)\bigr).
\end{equation}
Using that $\llbracket \varkappa c_z\rrbracket=0$  on  $z=h(t)$ then yields
$$
c_z\bigl(t,h(t)-\bigr)=\frac{\kappa_L}{\kappa}c_z\bigl(t,h(t)+\bigr)=
\frac{\kappa_L}{\kappa}\frac{c_*-c\bigl(t,h(t)\bigr)}{L}
$$
and one obtains the Robin type boundary condition
\begin{equation}\label{e1}
c\bigl(t,h(t)\bigr)+\frac{L\kappa}{\kappa_L}c_z\bigl( t,h(t)\bigr)=c_*,
\end{equation}
for $c\big |_{[0,h(t)]}$ where here and in the following
$c_z\bigl(t,h(t)\bigr)=c_z\bigl(t,h(t)-\bigr)$. In summary, assuming
that $\varepsilon_L=0$ we obtain the following reduced model
\begin{subequations}\label{EBBx}
\begin{align}
\varepsilon c_t&= \kappa c_{zz}-r(c),  &&0<z<h(t),\quad t>0,\label{EBBx2}\\
c\bigl(t,h(t)\bigr)+\frac{L\kappa}{\kappa_L}c_z\bigl(t,h(t)\bigr)&=c_*,\quad
c_z(t,0)=0, &&t>0,\label{EBBx4}\\
c(0,z)&=c_0(z), && z\in (0,h_0),\label{EBBx5}\\
h_t&=\int_0^h g\big(r(c(t,z))\big)\,\rd z,&& t>0,\label{EBBx6}\\
h(0)&=h_0.\label{EBBx7}
\end{align}
 \end{subequations}
on which we shall focus in the following. Throughout, if not stated
otherwise, we shall assume that
\begin{equation}\label{GeneralAssumptions}
g\in {\rm C}^{1-}(\R), \qquad r\in
{\rm C}^{1}(\R)\ \text{ with }\ sr(s)> 0\ \text{ for }\ s\not= 0,
\end{equation}
where ${\rm C}^{1-}$ refers to Lipschitz continuity. In fact, in all of our analysis only the restrictions of $r$ to $[0,c_*]$ and of $g$ to  $[0,r(c_*)]$ play a role. Note that
\eqref{GeneralAssumptions} implies $r(0)=0$.

\section{The Evolution Problem}
In this section we focus on the full evolutionary system
\eqref{EBBx} and establish the existence of a unique
global strong solution. We then derive some qualitative properties of
the solution including the asymptotic extinction of the biofilm when
the growth rate is negative, the boundedness of the substrate
gradient for initial data with bounded gradient as well as the
preservation of susbstrate monotonicity during the evolution.  
\subsection{Well-Posedness of the Evolution Problem}

We consider the evolution problem~\eqref{EBBx} with  $\varepsilon=1$
for notational simplicity. Introducing dimensionless variables via
\begin{equation}\label{trafo}
y=\frac{z}{h(t)} \quad\text{ and } \quad v(t,y):=c_*-c\bigl(t,yh(t)\bigr),
\end{equation}
the moving boundary problem \eqref{EBBx} with  $\varepsilon=1$ is
equivalent to the following initial (fixed) boundary value problem
with homogeneous boundary conditions
\begin{subequations}
\label{vEBBB}
\begin{align}
&v_t= \frac{\kappa}{h^2(t)}v_{yy}+\frac{h_t(t)}{h(t)} y v_y + r\bigl(
c_*-v(t,y)\bigr),  &&0<y<1,\quad t>0,\label{vEBBB2}\\ 
&\frac{\kappa}{h(t)^2}v_y(t,0)=0, &&t>0,\label{vEBBB3}\\
&\frac{\kappa_L}{Lh(t)}v(t,1)+\frac{\kappa}{h(t)^2} v_y(t,1)=0, &&t>0,\label{vEBBB4}\\
&v(0,y)=v_0(y):=c_*-c_0(yh_0), && 0<y<1,\label{vEBBB5}\\
&h_t(t)=h(t)\int_0^1 g\big(r\bigl(c_*- v(t,y)\bigr)\big)\,\rd y,&& t>0,\label{vEBBB6}\\
&h(0)=h_0,\label{vEBBB7}
\end{align}
\end{subequations}
where we write the boundary conditions \eqref{vEBBB3}-\eqref{vEBBB4}
in this particular form to be consistent with the notation used in \cite{Amann_Teubner}. 

\begin{thm}\label{T10}
Assume~\eqref{GeneralAssumptions} and let $p\in
(1,\infty)$. Given  initial values $h_0\in (0,\infty)$ and $v_0\in \operatorname{W}_p^1\bigl(
(0,1)\bigr)$ with $v_0(y)\in [0,c_*]$ for $y\in (0,1)$, there is a unique
global solution $(v,h)$ of~\eqref{vEBBB}  such that
$$
h\in {\rm C}^1\big([0,\infty),(0,\infty)\big) 
$$
and
$$
v\in {\rm C}\bigl([0,\infty), \operatorname{W}_p^1\bigl( (0,1)\bigr) \bigr) \cap
{\rm C}\big((0,\infty), \operatorname{W}_p^2\big( (0,1)\bigr)\big)\cap
{\rm C}^1\big([0,\infty), {\rm L}_p\big( (0,1)\bigr)\big) 
$$
with $v(t,y)\in [0,c_*]$ for $t\ge 0$ and $y\in (0,1)$. 
\end{thm}

\begin{proof}
{\bf Step 1.} In order to establish the local existence of a solution we recast~\eqref{vEBBB} as a quasilinear equation for $w=(v,h)$ to
which the theory described in \cite{Amann_Teubner} applies. To this
end, define (formally) the differential operator
$$
\mathcal{A}(w):= \frac{\kappa}{h^2} \partial_y^2 +G(v)\,  y \,\partial_y 
$$
acting on functions defined on $\Omega:=(0,1)$ with
$$
G(v):=\int_0^1 g\big(r\big(c_*-v(y)\big)\big)\,\rd y
$$
and the boundary operator
$$
\mathcal{B}(w)u:=\begin{cases} \frac{\kappa}{h^2}\partial_y u(0), &
  y=0,\\[1ex]
 \frac{\kappa_L}{Lh}u(1)+\frac{\kappa}{h^2} \partial_yu(1), & y=1,
\end{cases} 
$$
 on $\partial\Omega=\{0,1\}$.
If $w=(v,h)\in {\rm C}(\overline{\Omega})\times (0,\infty)$, then
$\bigl(\mathcal{A}(w),\mathcal{B}(w)\bigr)\in \mathcal{E}^1(\Omega)$ due to
\cite[Example~4.3~(e)]{Amann_Teubner}; that is,
$\bigl(\mathcal{A}(w),\mathcal{B}(w)\bigr)$ is normally elliptic on
$\Omega$. In fact, we have 
\begin{equation}\label{AB}
(\mathcal{A},\mathcal{B})\in
{\rm C}^{1-}\bigl({\rm C}(\overline{\Omega})\times
(0,\infty),\mathcal{E}^1(\Omega)\bigr). 
\end{equation}
For ${\rm S}\in\{\operatorname{H},\operatorname{W}\}$ we define the spaces (see
\cite[Section~7]{Amann_Teubner} with $\delta=1$ in the notation
therein, cf. \cite[Equation (7.5)]{Amann_Teubner}) 
$$
\operatorname{S}_{p,\mathcal{B}(w)}^s(\Omega):=\begin{cases} \big\{u\in
  \operatorname{S}_{p}^s(\Omega)\,\vert\, \mathcal{B}(w) u=0 \ \text{on }
  \{0,1\}\big\}, &\text{if }s\in(1+\frac{1}{p},2],\\[1ex]
 \operatorname{S}_{p}^s(\Omega) , &\text{if }s\in(-1+\frac{1}{p},1+\frac{1}{p}),\\[1ex]
\big( \operatorname{S}_{p'}^{-s}(\Omega)\big)' , &\text{if
}s\in(-2+\frac{1}{p},-1+\frac{1}{p}].\\[1ex]
\end{cases} 
$$
It is an important observation that
$\operatorname{S}_{p,\mathcal{B}(w)}^s=\operatorname{S}_{p,\mathcal{B}}^s$ is independent of $w$ for
$s\in(-2+\frac{1}{p}, 1+\frac{1}{p})$. We now fix $2\sigma\in
(1,1+\frac{1}{p})$ and denote the $\operatorname{H}_{p,\mathcal{B}}^{2\sigma-2}$-realization of
$(\mathcal{A}(w),\mathcal{B}(w))$ by $A_{\sigma-1}(w)$. Using
\eqref{AB}, we can invoke \cite[Theorem~8.3]{Amann_Teubner} to see that
\begin{equation}\label{Asigma}
  A_{\sigma-1}\in{\rm C}^{1-}\big(
  {\rm C}(\overline{\Omega})\times
  (0,\infty),\mathcal{H}\bigl(\operatorname{H}_{p,\mathcal{B}}^{2\sigma}(\Omega),
  \operatorname{H} _{p,\mathcal{B}}^{2\sigma-2}(\Omega)\bigr)\big). 
\end{equation}
We then put
$$
E_0:=\operatorname{H}_{p,\mathcal{B}}^{2\sigma-2}(\Omega)\times \R \quad \text{ and }\quad
E_1:=\operatorname{H} _{p,\mathcal{B}}^{2\sigma}(\Omega)\times \R 
$$
and let $E_\theta:=(E_0,E_1)_{\theta,p}$ be the real interpolation space if $\theta\in (0,1)\setminus\{1-\sigma,\frac{3}{2}-\sigma\}$ respectively $E_\theta:=[E_0,E_1]_{\theta}$ the complex  interpolation space if $\theta\in \{1-\sigma,\frac{3}{2}-\sigma\}$; that is,
$$
E_\theta=\operatorname{W}_{p,\mathcal{B}}^{2\sigma-2+2\theta}(\Omega)\times
\R,\quad \theta\in(0,1).
$$
Let $2\alpha:=3-2\sigma\in (2-\frac{1}{p},2)$ and $2\beta\in
(2+\frac{1}{p}-2\sigma,2\alpha)$. Then
$E_\alpha=\operatorname{W}_{p,\mathcal{B}}^{1}(\Omega)\times \R$ and $E_\beta
\hookrightarrow {\rm C}(\overline{\Omega})\times\R$. Setting
$$
\mathbb{A}(w):=\left[\begin{matrix}
A_{\sigma-1}(w)&0\\ 0 & G(v)
\end{matrix}\right],
$$
we thus obtain that
\begin{equation}\label{AA}
\mathbb{A}\in {\rm C}^{1-}\big(X_\beta,\mathcal{H}(E_1,E_0)\big)
\end{equation}
with the open subset
$X_\beta:=\operatorname{W}_{p,\mathcal{B}}^{2\beta-2+2\theta}(\Omega)\times
(0,\infty)$ of $E_\beta$. Finally, setting
$$ f(w):=\big(r(c_*-v),0\big) 
$$
we have that
\begin{equation}\label{fLower}
f\in {\rm C}^{1-}\big(X_\beta,E_\gamma\big)
\end{equation}
where $E_\gamma={\rm L}_p(\Omega)\times \R$ with
$2\gamma:=2-2\sigma \in (1-\frac{1}{p},2-\frac{1}{p})$. Consequently,
problem~\eqref{vEBBB} can be written as the quasilinear abstract Cauchy problem
\begin{equation}\label{CP}
\dot w=\mathbb{A}(w)w +f(w),\  t>0,\quad w(0)=w_0:=(v_0,h_0).
\end{equation}
Since $0<\gamma<\beta<\alpha<1$, it follows from \eqref{AA},
\eqref{fLower}, and \cite[Theorem~12.1]{Amann_Teubner} that for each
$w_0\in \operatorname{W}_{p,\mathcal{B}}^{1}(\Omega)\times (0,\infty)$
there is a unique maximal (weak) solution $w=(v,h)$ of \eqref{CP} on a
maximal interval $J=\bigl[0,t^+(w_0)\bigr)$ satisfying (with $\dot J:=J\setminus\{0\})$
\begin{equation}\label{w}
w=w(\cdot,w_0)\in {\rm C}\big(J,E_\alpha\big) \cap {\rm C}\big(\dot
J, E_1\big)\cap {\rm C}^1\big(J, E_0\big),
\end{equation}
and the mapping $\bigl[(t,w_0)\mapsto w(t,w_0)\bigr]$ defines a
semiflow on ${\rm W}_{p,\mathcal{B}}^{1}(\Omega)\times
(0,\infty)$. Whenever $t^+(w_0)<\infty$, it must hold that
\begin{equation}\label{globalex}
h(t)\to 0=\partial (0,\infty)\ \text{ or }\
\|v(t)\|_{\operatorname{W}_p^1}+h(t)\to\infty
\end{equation}
as $t\nearrow t^+(w_0)$. Since the right-hand side of
$$
h_t(t)=h(t)\int_0^1 g\big( r\bigl( v(t,y)+c_*\bigr)\big)\,\rd y
$$
belongs to ${\rm C}(J,\R)$ as a function of $t$, we have
in fact that $h\in {\rm C}^1\bigl( J,(0,\infty)\bigr)$.\\

{\bf Step 2.} In order to improve the regularity of $v$ choose
$1-\frac{1}{p}<2\nu<2\sigma-\frac{1}{p}$. This is possible since
$2\sigma\in (1,1+\frac{1}{p})$. Then we infer from \eqref{w} that
\begin{align}
v\in {\rm C}\big(\dot J,\operatorname{H}_p^{2\sigma}(\Omega)\big)
\cap{\rm C}^1\big(\dot J,\operatorname{H}_p^{2\sigma-2}(\Omega)\big)
  &\hookrightarrow  {\rm C}^\nu\big(\dot J,
\operatorname{W} _p^{2\sigma-2\nu}(\Omega)\big)\hookrightarrow
{\rm C} ^\nu\big(\dot J,{\rm C}(\overline{\Omega})\big)\label{vv}
\end{align}
so that
$$
G(v)=\left[ t\mapsto\int_0^1 g\big( r\bigl( c_*-v(t,y)
  \bigr)\big)\,\rd y\right]\in {\rm C}^\nu(\dot J,\R).
$$
We conclude that
\begin{equation}\label{AAA}
\Big[ t\mapsto \Big( \mathcal{A}\bigl( w(t) \bigr),\mathcal{B}\bigl(
w(t) \bigr)\Big)\Big]\in{\rm C}^{\nu}\big(\dot
J,\mathcal{E}^1(\Omega)\big),\quad 2\nu>2-1-\frac{1}{p}.
\end{equation}
Clearly, \eqref{vv} also entails that 
\begin{equation}\label{fUpper}
F\in {\rm C}^\nu\big(\dot J,{\rm L}_p(\Omega)\big)
\end{equation}
for  $F(t):=r\bigl( c_*-v(t,\cdot)\bigr)$. Consequently, it follows
from \eqref{AAA}, \eqref{fUpper}, and
\cite[Theorem~11.3]{Amann_Teubner} (see also
\cite[Theorem~13.3]{Amann_Teubner}) that
$$
v\in {\rm C}\big(J, \operatorname{W}_p^1(\Omega)\big) \cap
{\rm C}\big(\dot J, \operatorname{W} _p^2(\Omega)\big)\cap
{\rm C}^1\big(\dot J, {\rm L} _p(\Omega)\big) 
$$
satisfies
\begin{alignat*}{2}
v_t(t)&= \mathcal{A}\bigl( w(t)\bigr)v(t)+F(t), \quad && t\in \dot J,\\
\mathcal{B}\bigl( w(t)\bigr)v(t)&=0, && t\in \dot J.
\end{alignat*}
Hence, $v$ has the  regularity
properties stated in the assertion.\\ 

{\bf Step 3.} Next we establish that $0\le v(t,y)\le c_*$ for $t\ge 0$ and $y\in [0,1]$.  Note that $v$ satisfies
\begin{alignat*}{2}
v_t(t)&= \mathcal{A}(t)v(t)+r(c_*-v), \quad&&t\in \dot J,\\
\mathcal{B}(t) v(t)&=0, && t\in \dot J,\\
 v(0)&=v_0,&&
\end{alignat*}
where $(\mathcal{A}(t),\mathcal{B}(t)):=\bigl(\mathcal{A}\bigl( w(t)
\bigr),\mathcal{B}\bigl( w(t) \bigr)\bigr)$ and $r\in
{\rm C}^1(\R)$. 
Now, since $v_0\geq 0$ and $r(c_*)>0$, we infer
from~\cite[Theorem~15.1]{Amann_Teubner} (the assumptions
(15.1)-(15.3) therein are easily checked with $N=1$ and $\delta=1$)
that $v(t,\cdot)\ge 0$ for $t\in J$. Similarly, setting $u:=c_*-v$ we have
\begin{alignat*}{2}
u_t(t)&= \mathcal{A}(t)u(t)-r(u), \quad &&t\in \dot J,\\
\mathcal{B}(t)u(t)&=c_*, &&t\in \dot J,\\
u(0)&=u_0:=c_*-v_0,&&
\end{alignat*}
with $c_*> 0$, $r(0)=0$, and $u_0\ge 0$ and it again follows from
\cite[Theorem~15.1]{Amann_Teubner}  that $u(t,\cdot)\ge 0$ for $t\in
J$. Therefore, we indeed have $0\le v(t,y)\le c_*$ for $t\in J$ and
$y\in [0,1]$.\\

{\bf Step 4.} Finally, we prove that the solution exists globally;
that is, that we always have $J=[0,\infty)$. To this end, we use
the just established bound $0\le v(t,y)\le c_*$ for $t\in J$ and $y\in
[0,1]$ to derive the existence of $M>0$ such that
\begin{equation}\label{GG}
\vert G(v(t))\vert\le M, \quad t\in J.
\end{equation} 
Assume towards a contradiction that
$t^+(w_0)<\infty$. Then~\eqref{vEBBB6} implies that
\begin{equation}\label{hh}
e^{-Mt}h_0\le h(t)\le e^{Mt}h_0,\quad t\in J,
\end{equation}
and hence 
$$
\vert h(t)-h(s)\vert=\left\vert\int_s^t h(\tau) G(v(\tau))\,\rd
  \tau\right\vert\le e^{Mt^+(w_0)}h_0 M\vert t-s\vert,\quad t,s\in J.
$$
That is, 
\begin{equation}\label{hhh}
h\in {\rm BUC}^{1-}(J,(0,\infty)).
\end{equation}
For $t\in J$ define the differential operator $\hat{\mathcal{A}}(t):= \frac{\kappa}{h(t)^2} \partial_y^2$
in $\Omega=(0,1)$ and the boundary operator $\hat{\mathcal{B}}(t)$
 by
$$
\hat{\mathcal{B}}(t)u:=\begin{cases} \frac{\kappa}{h(t)^2}u_y(0), & y=0,\\[1ex]
 \frac{\kappa_L}{Lh(t)}u(1)+\frac{\kappa}{h(t)^2} u_y(1), & y=1,
\end{cases} 
$$
on $\partial\Omega=\{0,1\}$. Fix $2\sigma\in (1,1+1/p)$ and let
$\hat{A}(t)$ be the
$\operatorname{H}_{p,\mathcal{B}}^{2\sigma-2}$-realization of
$\bigl( \hat{\mathcal{A}}(t),\hat{\mathcal{B}}(t)\bigr)$. Then, as in {\bf Step 1}, we deduce from \cite[Theorem~8.3]{Amann_Teubner}
and \eqref{hhh}  that
\begin{equation}\label{Ahatsigma}
\hat{A}\in
{\rm BUC}^{1-}\big(J,\mathcal{H}(\hat{E}_1,\hat{E}_0)\big),
\end{equation}
where
$$
\hat{E}_0:=\operatorname{H}_{p,\hat{\mathcal{B}}}^{2\sigma-2}(\Omega)
\quad\text{ and } \quad\hat{E}_1:=\operatorname{H}_{p,\hat{\mathcal{B}}}^{2\sigma}(\Omega) 
$$
are independent of $t\in J$ since
$-2+\frac{1}{p}<2\sigma-2<2\sigma<1+\frac{1}{p}$ 
and coincide with the corresponding spaces introduced in {\bf Step 1}. Set 
$$
\hat{f}(t,v):=G\bigl( v(t) \bigr)y v_y+r(c_*-v(t)), \quad t\in J,
$$
and, using \eqref{GG}, observe that the linear bound
\begin{align}\label{fest}\notag
\|\hat{f}(t,v)\|_{\hat{E}_\gamma}&\le \|G\bigl( v (t) \bigr) y v_y\|_{{\rm L}_p} +
\|r\bigl(c_*- v(t) \bigr)\|_\infty\\&\le M\|v\|_{\operatorname{W}_p^1}+
\|r\|_{\infty,[0,c_*]}\le m_0(1+\|v\|_{\hat{E}_\alpha})
\end{align} 
holds for $t\in J$, where $0<2\gamma=2-2\sigma<2\alpha=3-2\sigma<1$ as
 in {\bf Step 1} so that
$\hat{E}_\alpha=\operatorname{W}_p^1\bigl( (0,1) \bigr)$ and
$\hat{E}_\gamma={\rm L} _p\bigl( (0,1) \bigr)$. Since
$v$ satisfies the abstract semilinear Cauchy problem 
\begin{align*}
  v_t(t)=\hat{A}(t)v(t) + \hat{f}(t,v(t)),\quad  t\in \dot J,\qquad v(0)=v_0,
\end{align*}
we conclude from \eqref{Ahatsigma} and \eqref{fest} that there is some
$C=C\bigl(t^+(w_0)\bigr)>0$ such that 
$$
\|v(t)\|_{\operatorname{W}_p^1}=\|v(t)\|_{\hat{E}_\alpha}\le C,
\quad t\in J.
$$
Thanks to \eqref{hh} we deduce that \eqref{globalex} cannot occur. This
contradicts the assumption that~$t^+(w_0)<\infty$. Consequently, the
solution must exist globally in time; that is, we have
that~$t^+(w_0)=\infty$. This concludes the proof.
\end{proof}

\subsection{Qualitative Aspects}

The subsequent result says that the biofilm disappears over time for a
negative biofilm growth rate. In this case, asymptotically the domain
\mbox{$[0,h(t)+L]$} collapses to the diffusive bulk layer $[0,L]$ with
uniform substrate concentration~$c=c_*$ in it (recall~\eqref{trafo}).

\begin{prop}\label{P5}
Assume~\eqref{GeneralAssumptions} and that $g(r(s))< 0$ for $s\in
[0,c_*]$. Given $p\in (1,\infty)$ and initial values $h_0\in
(0,\infty)$ and $v_0\in {\rm W}_p^1\bigl(
(0,1)\bigr)$ with $v_0(y)\in [0,c_*]$ for $y\in (0,1)$ let  $(v,h)$
be the global solution of~\eqref{vEBBB}. Then 
$$
h(t)\searrow 0\quad \text{in }\ \R
$$
and
$$
v(t,\cdot)\rightarrow 0\quad \text{ in }\ {\rm L}_p\big((0,1)\big)
$$
as $t\rightarrow \infty$.
\end{prop}

\begin{proof}
We may assume that $p\ge 2$ due to the regularity of $v$. 
Recall from Theorem~\ref{T10} that $0\le v(t,y)\le c_*$ for $t\ge 0$
and $y\in [0,1]$ and, consequently, that $G(v(t))< 0$ for $t\ge 0$ by
assumption. Therefore, $h(t)\searrow 0$ as $t\to\infty$
by~\eqref{vEBBB6}. We then test~\eqref{vEBBB2} by $v$ and integrate by
parts, using the boundary conditions~\eqref{vEBBB3}-\eqref{vEBBB4}, to
obtain
\begin{align*}
  \frac{\rd}{\rd t} \frac{1}{2}\|v(t)\|_{{\rm L}_2(0,1)}^2&=
  \frac{\kappa}{h^2(t)}\int_0^1 vv_{yy}\,\rd y+G(v(t))\int_0^1 y
  \partial_y\left(\frac{1}{2}v^2\right)\,\rd y+ \int_0^1 v r(c_*-v)\,\rd y\\
  &=-\frac{\kappa}{h^2(t)}\int_0^1 \vert v_{y}\vert^2\,\rd y -
  \frac{\kappa_L}{Lh(t)}\vert v(t,1)\vert^2\\
  &\qquad+\frac{G(v(t))}{2}\vert v(t,1)\vert^2-\frac{G(v(t))}{2}
  \int_0^1\vert v\vert^2\,\rd y+ \int_0^1 v r(c_*-v)\,\rd y. 
\end{align*}
Since 
$$
-\|g\|_{\infty}:=-\|g\|_{\infty,[0,r(c_*)]}\le  G(v(t))\le 0, \qquad
0\le v r(c_*-v)\le c_* \|r\|_{\infty,[0,c_*]}=: R_0/2,
$$
it follows that
\begin{align*}
\frac{\rd}{\rd t} \frac{1}{2}\|v(t)\|_{{\rm L}_2(0,1)}^2&\leq
-\frac{1}{2}\min\left\{ \frac{\kappa}{h^2(t)},\frac{\kappa_L}{Lh(t)}\right\}
\left( 2\int_0^1\vert v_{y}\vert^2\,\rd y + 2\vert v(t,1)\vert^2\right)\\ 
&\phantom{\leq}+\frac{\|g\|_{\infty}}{2}\int_0^1\vert v\vert^2\,\rd y+
R_0/2 .
\end{align*}
Taking into account that
$$
\vert v(y)\vert^2\le 2\vert v(1)\vert^2+2\int_y^1\vert
v_y(s)\vert^2\,\rd s,\quad y\in (0,1),
$$
so that
$$
\|v\|_{{\rm L}_2(0,1)}^2\le 2\left(\vert
  v(1)\vert^2+\|v_y\|_{{\rm L}_2(0,1)}^2\right)
$$
we derive that
\begin{align*}
\frac{\rd}{\rd t} \|v(t)\|_{{\rm L}_2(0,1)}^2&\le 
\left(\|g\|_{\infty}-\min\left\{
\frac{\kappa}{h^2(t)},\frac{\kappa_L}{Lh(t)}\right\}\right)
\|v(t)\|_{{\rm L}_2(0,1)}^2 +R_0.  
\end{align*}
Given $q>0$ arbitrary, there is $t_0>0$ such that 
$$
\|g\|_{\infty}-\min\left\{\frac{\kappa}{h^2(t)},\frac{\kappa_L}{Lh(t)}\right\}\le 
-q,\quad t\ge t_0.
$$
Consequently,
$$
\|v(t)\|_{{\rm L}_2(0,1)}^2\le
e^{-q(t-t_0)}\|v(t_0)\|_{{\rm L} _2(0,1)}^2 +
\frac{1}{q}\left(1-e^{-q(t-t_0)}\right)R_0, \quad t\ge t_0,
$$
and hence
$$
\limsup_{t\to\infty}\|v(t)\|_{{\rm L}_2(0,1)}^2\le\frac{R_0}{q}.
$$
Since $q>0$ was arbitrary, we deduce that
$$
\lim_{t\to\infty}\|v(t)\|_{{\rm L}_2(0,1)}=0.
$$
Noticing that
$$
\|v(t)\|_{{\rm L}_p(0,1)}\le c_*^{1-2/p}
\|v(t)\|_{{\rm L}_2(0,1)}^{2/p}, \quad t>0, 
$$
the assertion follows.
\end{proof}

We next provide gradient estimates for the substrate.

\begin{prop}\label{P6}
Assume~\eqref{GeneralAssumptions} and let $r'\ge 0$. Given  $h_0\in
(0,\infty)$ and \mbox{$v_0\in \operatorname{W}_\infty^1\bigl(
(0,1)\bigr)$} with \mbox{$v_0(y)\in [0,c_*]$} for $y\in (0,1)$, let
$(v,h)$ be the global solution of~\eqref{vEBBB}. Then there exists 
$M>0$ such that 
\begin{equation}\label{bound}
\| v_y(t,\cdot)\|_{\infty}\leq M h(t), \quad t\geq 0.
\end{equation}
Moreover, if
$v_0$ is non-increasing, then
$$
- M h(t)\le v_y(t,y)\le 0, \quad y\in [0,1],\ t\geq 0.
$$
\end{prop}

\begin{proof}
First observe from~\eqref{EBBx} that $d=c_z$ satisfies the equation
$$\begin{cases}
  d_t=\kappa d_{zz}-r'(c)d&\text{in }\bigl( 0,h(t)\bigr)\text{ for } t>0,\\
  d(t,0)=0,\quad d\bigl( t,h(t)\bigr)=\frac{\kappa _L}{\kappa L}(c_*-c)&
  \text{for }t>0,\\d(0,\cdot)=(c_0)_z&\text{in }(0,h_0).\end{cases}
$$
Since $r'(c)\ge 0$, it follows from the comparison principle that
$$
\| d(t,\cdot)\|_\infty\leq M:=\max\left\{ \| (c_0)_z\|_\infty, \frac{c_*\kappa}{\kappa
  L}\right\}, \quad t\geq 0.
$$
From this and~\eqref{trafo} we infer \eqref{bound}. Finally, if
$v_0$ is non-increasing, then $(c_0)_z\ge 0$ so that $d$ is
nonnegative on the ``parabolic boundary'' thanks to the fact that
$0\leq c\leq c_*$ and the boundary condition satisfied at
$z=h(t)$. The weak maximum principle thus ensures that $d(t,\cdot)\geq
0$ and the assertion follows from~\eqref{trafo}.
\end{proof}

Observe that~\eqref{bound} implies that, in Proposition \ref{P5}, we
have  \mbox{$v(t,\cdot)\to 0$ in $\operatorname{W}^1_p(0,1)$} as
$t\to\infty$ for any $p\in (1,\infty)$ whenever $v_0\in
\operatorname{W}_\infty^1\bigl( (0,1)\bigr)$.


\section{Quasi-Steady Approximation}
In this section we give a more detailed account of the dynamics 
focusing on a quasi-steady approximation of~\eqref{EBBx}. 
If diffusion inside the biofilm is also assumed to take place
at a much faster time scale than biofilm growth (as is observed
experimentally), it can be assumed
that $\varepsilon =0$ in \eqref{EBBx} and the system
becomes
\begin{subequations}
\label{BBx}
\begin{align}
\kappa c_{zz}&= r(c),  &&0<z<h(t),\quad t>0,\label{BBx2}\\
c\bigl(t,h(t)\bigr)+\frac{L\kappa}{\kappa_L}c_z\bigl(t,h(t)\bigr)&=c_*,\quad
c_z(t,0)=0, &&t>0,\label{BBx4}\\
h_t&=\int_0^h g\big(r(c(t,z))\big)\,\rd z,&& t>0,\label{BBx6}\\
h(0)&=h_0.\label{BBx7}
\end{align}
 \end{subequations}
For this simplified model, we are able to provide a fairly complete
picture of the qualitative aspects of the dynamics under suitable
assumptions on the biofilm growth rate $g$ and the substrate
consumption rate $r$.

As in the previous section it is convenient to introduce dimensionless
variables via
$$
y=\frac{z}{h(t)}\quad \text{ and }\quad u(t,y)=c\bigl(t,yh(t)\bigr),
$$
so that the moving boundary problem \eqref{BBx} is equivalent to the
following initial (fixed) boundary value problem
\begin{subequations}
\label{BBB}
\begin{align}
\frac{\kappa}{h^2(t)} u_{yy}&= r(u),  &&0<y<1,\quad t>0,\label{BBB2}\\
u_y(t,0)&=0, &&t>0,\label{BBB4}\\
u(t,1)+\frac{L\kappa}{\kappa_L h(t)} u_y(t,1)&=c_*, &&t>0,\label{BBB3}\\
h_t(t)&=h(t)\int_0^1 g\big(r(u(t,y))\big)\,\rd y,&& t>0,\label{BBB6}\\
h(0)&=h_0.\label{BBB7}
\end{align}
\end{subequations}
In order to analyze the dynamics we establish useful properties of the solution $u[h]$ of the sub-problem
\eqref{BBB2}-\eqref{BBB3}  as a function of $h$. Existence of a smooth solution
$u[h]$ is obtained by means of a fixed point argument. Then the system
is reduced to an ODE for~$h$ and solved by standard arguments. This
allows us to characterize the behavior of the ``extreme'' solutions
observed when $h=0$ or $h=\infty$. We also establish the
existence of non-trivial equilibria for general nonlinearities $g$ and
$r$ (on the already mentioned minimal regularity and physical
assumptions). In the special case of affine $g$, we can ensure the
uniqueness of the non-trivial equilibrium and prove that the solution
of the initial value problem converges to it.

\subsection{Well-Posedness}

We first establish the existence of a unique solution to~\eqref{BBB}.

\begin{thm}\label{T1}
Assume \eqref{GeneralAssumptions} with $r'\ge 0$. Then, for each $h_0>0$ there is a
unique global solution
$$
(h,u)\in {\rm C}^1\big([0,\infty),(0,\infty)\times
{\rm C}^2([0,1])\big)
$$
to \eqref{BBB}. Moreover, it holds that $0\le u(t,y)\le c_*$ for $t\ge
0$ and $y\in[0,1]$.
\end{thm}

The proof of Theorem \ref{T1} is a consequence of the following
proposition in which we study the
sub-problem~\eqref{BBB2}-\eqref{BBB3} for fixed $h$.

\begin{prop}\label{P1}
Assume \eqref{GeneralAssumptions} with $r'\ge 0$. For each $h>0$ there exists a
unique solution $$u=u[h]\in {\rm C}^2 \bigl( [0,1]\bigr)$$ of the
boundary value problem
\begin{subequations}\label{BBBB}
\begin{align}
\frac{\kappa}{h^2} u_{yy}&= r(u),\quad  0<y<1,\label{BBBB1}\\
u_y(0)&=0,\quad u(1)+\frac{L\kappa}{\kappa_L h}u_y(1)=c_*.\label{BBBB2}
\end{align}
\end{subequations}
It holds that
\begin{align}\label{C}
  u(y)=c_*-\frac{Lh}{\kappa_L}\int_0^1
  r \bigl( u(\eta)\bigr)\,\rd \eta-\frac{h^2}{\kappa}
  \int_y^1\int_0^\eta r \bigl( u(\rho)\bigr)\,\rd\rho \,\rd \eta,\quad y\in [0,1].
\end{align}
Moreover, one also has that
\begin{subequations}\label{ccc}
\begin{align}
&0< u(0)\le u(y)\le u(1)< c_*,\quad  y\in [0,1],\label{ccc1}\\
&0\le u_y(y)\le u_y(1)=\frac{\kappa_L}{L\kappa}h \big(c_*-u(1)\big),\quad 
y\in [0,1],\label{ccc2}\\
&0\le u_{yy}(y)\le \frac{r(c_*)}{\kappa}h^2,\quad  y\in [0,1], \label{ccc3}
\end{align}
\end{subequations}
and
$$
\Bigl[ h\mapsto u[h]\Bigr]\in {\rm C}^1\big((0,\infty),
{\rm C}^2([0,1])\big).
$$
\end{prop}

\begin{proof}
{\bf Step 1.} We define
$$
F(u)(y):=c_*-\frac{Lh}{\kappa_L}\int_0^1
  \tilde r \bigl( u(\eta)\bigr)\,\rd \eta-\frac{h^2}{\kappa}
  \int_y^1\int_0^\eta \tilde r \bigl( u(\rho)\bigr)\,\rd\rho \,\rd \eta,\quad y\in [0,1],
$$
for $u\in {\rm C}([0,1])$, where
$$
\tilde r (s):=
\begin{cases}
 r(s),& s\in[0,c_*],\\ 0,&s< 0,\\ r(c_*),& s\geq  c_*,
\end{cases}
$$
which will prove no restriction once we establish \eqref{ccc1}. The
function $F$ clearly satisfies
$$
F: {\rm C}([0,1])\rightarrow {\rm C}([0,1])
$$
by definition using the fact that the function $\tilde r$ is
continuous. Thanks to
$$
\partial_y F(u)(y)=\frac{h^2}{\kappa} \int_0^y \tilde r\bigl(
u(\rho)\bigr)\,\rd \rho,\quad y\in [0,1],
$$
for  $u\in {\rm C}([0,1])$ and since $\tilde r$ is bounded,
the Arzel\`a-Ascoli theorem ensures that the mapping $F:
{\rm C}([0,1])\rightarrow {\rm C}([0,1])$ is also
compact. The boundedness of $\tilde r$ also implies that the set
$$
\left\{u\in C([0,1])\, \big |\,  \text{there is } \lambda\in [0,1] \text { with }
u=\lambda F(u)\right\} 
$$
is bounded. Consequently, Sch\"afer's fixed point \cite[Theorem~11.3]{GT} theorem entails that
there is $u\in {\rm C}([0,1])$ with $u=F(u)$. In fact, since
$$
u_y(y)=\frac{h^2}{\kappa} \int_0^y \tilde r\bigl(u(\eta)\bigr)\,\rd
\eta,\quad y\in [0,1], 
$$
we also have that $u\in {\rm C}^2([0,1])$ and satisfies
\begin{subequations}\label{BBBB4}
\begin{align}\label{BBBB4a}
\frac{\kappa}{h^2} u_{yy}(y)&= \tilde r\bigl(u(y)\bigr),\quad  y\in [0,1],
\end{align}
as well as
\begin{equation}\label{BBBB4b}
u_y(0)=0, \qquad \frac{L\kappa}{\kappa_L h}u_y(1)=\frac{L h}{\kappa_L}
\int_0^1 \tilde r\bigl(u(\eta)\bigr)\,\rd \eta=c_*-u(1).
\end{equation}
\end{subequations}
Next we consider the additional estimates for the solution. Starting
with $u\geq 0$, first assume by contradiction that $u\le 0$
everywhere. Then $\tilde r(u)\equiv 0$ and \eqref{BBBB4a} implies that
$u_y\equiv 0$ since $u_y(0)=0$. However, according to \eqref{BBBB4b} it
must hold $u_y(1)>0$ since $u(1)\le 0$ by premise. Thus we have a
contradiction. It still cannot be excluded that $u(y_0)<0$ for some
$y_0\in [0,1]$, so, towards a contradiction, assume that this is the
case. Then there are $0\le y_1<y_2\le 1$ such that  $u(y)<0$ for
$y\in (y_1,y_2)$ and $u(y_1)=0$ or $u(y_2)=0$. Using again that
$s\tilde r(s)\ge 0$, this is impossible since \eqref{BBBB4a} implies that
$u$ is concave on $(y_1,y_2)$ and, consequently, $u\ge 0$ everywhere
and \eqref{BBBB4a} shows that that $u$ is convex. Notice that
$u\equiv0$ is excluded by \eqref{BBBB4b} given that $\tilde
r(0)=0$. We conclude that $u\geq 0$.

Next,  we note that $u(0)>0$ since otherwise $u\equiv 0$ as the
solution to~\eqref{BBBB1} with $u(0)=u_y(0)=0$. Moreover, if
$u(1)=c_*$, then $u_y(1)=0$ and the convexity of $u$ would imply that
$u(y)>c_*$ for $y<1$ close to 1, which is impossible. Hence,
exploiting that $u_y(0)=0$ and convexity we see that
$$
0< u(0)\le u(y)\le u(1)< c_*,\quad y\in (0,1),
$$
making use again of \eqref{BBBB4b} for the upper bound $c_*$. Thus
\eqref{ccc1} is established. It follows, in particular, that $\tilde r
\bigl( u(y)\bigr) =r \bigl( u(y)\bigr)$ for $y\in [0,1]$ and
\eqref{BBBB4} ensures that $u$ solves the unmodified problem
\eqref{BBBB}.

Thanks to the convexity of $u$ and the monotonicity of $r$ we also
derive \eqref{ccc2} and~\eqref{ccc3}.
 
Let us now turn to uniqueness. Take two solutions $u_1$ and $u_2$ of
\eqref{BBBB} and set $w:=u_1-u_2$. The latter satisfies
$$
\frac{\kappa}{h^2} w_{yy}= r(u_1)-r(u_2), \quad y\in [0,1]
$$
and
$$
w_y(0)=0,\quad w(1)+\frac{L\kappa}{\kappa_Lh}w_y(1)=0. 
$$
Multiplying the first equation with $w$ and integrating by parts
gives
\begin{align*}
0&\le\int_0^1\Big[r\bigl(u_1(y)\bigr)-r\bigl(u_2(y)\bigr)\Big]\Big[u_1(y)-u_2(y)\Big]\,\rd y\\
&=-\frac{\kappa}{h^2} \int_0^1 w_y^2(y)\,\rd y+ \frac{\kappa}{h^2} w(y)w_y(y)\Big\vert_{y=0}^{y=1}\\
&= -\frac{\kappa}{h^2} \int_0^1 w_y^2(y)\,\rd y -  \frac{L\kappa^2}{\kappa_L h^3} w_y^2(1)
\end{align*}
using the boundary conditions and the monotonicity of $r$. It follows
that $w\equiv 0$ on $[0,1]$.\\[.1cm]
{\bf Step 2.} It remains to show the continuous dependence of the
solution on $h>0$. We write~$u[h]$ for the unique solution just found
in {\bf Step 1} and set
\begin{equation*}
F(h,v)(y):=v(y)-c_*+\frac{Lh}{\kappa_1}\int_0^1r\bigl(v(\eta)\bigr)\,\rd \eta
+\frac{h^2}{\kappa} \int_y^1 \int _0^\eta r\bigl(v(\rho)\bigr)\,\rd
\rho\, \rd\eta
\end{equation*}
for $h>0$, $y\in [0,1]$, and $v\in {\rm C}^2([0,1])$. Then
$$
F:(0,\infty)\times {\rm C}^2([0,1])\rightarrow
{\rm C}^2([0,1])
$$ 
satisfies $F\bigl(h,u[h]\bigr)=0$ for every
$h>0$. Moreover, its Fr\'echet derivative is given by $D_v
F\bigl(h,u[h]\bigr)=1+K$, where
\begin{align*}
(K\varphi) (y):&=\frac{Lh}{\kappa_L}\int_0^1 r'\bigl( u[h](\eta)
\bigr)\varphi(\eta)\,\rd \eta+\frac{h^2}{\kappa} \int_y^1 \int_0^\eta
r'\bigl( u[h](\rho)\bigr)\varphi(\rho)\,\rd \rho \rd \eta
\end{align*}
for $h>0$, $y\in [0,1]$, and $\varphi\in
{\rm C}^2([0,1])$. Since
$$
\partial_y (K\varphi) (y)=-\frac{h^2}{\kappa} \int_0^y  r'\bigl(
u[h](\rho)\bigr)\varphi(\rho)\,\rd \rho\text{ and } 
\partial_y^2 (K\varphi) (y)=-\frac{h^2}{\kappa}  r'\bigl(
u[h](y)\bigr)\varphi(y),
$$
the Arzel\`a-Ascoli theorem ensures, as in {\bf Step 1}, that $K\in
\mathcal{L}({\rm C}^2([0,1]))$ is a compact operator, hence
$D_v F(h,u[h]) =1+ K$ is a Fredholm operator. If $\varphi\in \mathrm{ker}
\bigl(D_v F\bigl(h,u[h]\bigr)\bigr)$, then $\varphi\in
{\rm C}^2([0,1])$ satisfies
\begin{align*}
\frac{\kappa}{h^2}\varphi_{yy}(y)&= r'\bigl(u[h](y)\bigr)\varphi(y), \quad y\in [0,1],\\
\varphi_y(0)&=0,\quad \varphi(1)+\frac{L\kappa}{\kappa_L h}\varphi_y(1)=0.
\end{align*}
Since $r'\bigl( u[h](y)\bigr )\ge 0$ for $y\in [0,1]$, it follows as
in the uniqueness proof that $\varphi\equiv 0$ on~$[0,1]$. This means
that $D_v F\bigl( h,u[h]\bigr)$ is an automorphism on
${\rm C}^2([0,1])$ for each $h>0$. Since $F\bigl(
h,u[h]\bigr)=0$ for $h>0$, the implicit function theorem finally
ensures that indeed 
$$
\big[ h\mapsto u[h] \big]\in {\rm C}^1\bigl((0,\infty),
{\rm C}^2([0,1])\big),
$$
and the proof the proposition is complete.
\end{proof}

\begin{proof}[Proof of Theorem~\ref{T1}]
We set
\begin{equation}
\label{f}
f(h):=h\int_0^1 g\bigl(r\bigl(u[h](y)\bigr)\bigr)\,\rd y,\quad h>0,
\end{equation}
where $u[h]$ denotes the unique solution  of \eqref{BBBB}. From
Proposition~\ref{P1} we infer that \mbox{$f\in {\rm
    C}^1\big((0,\infty),\R)$} with $\vert f(h)\vert\le Mh$ for
$M=\|g\|_{\infty,[0,c_*]}$. Therefore, for every given $h_0>0$ there
is a unique global solution $h\in {\rm C}^1\bigl([0,\infty)\bigr)$
of the ODE
$$ 
h'=f(h),\quad  t>0,\qquad h(0)=h_0,
$$
with $h(t)\ge \exp(-M t)h_0$ for $t\ge 0$. Proposition \ref{P1} then
ensures that
$$
\Bigl[ t\mapsto u\bigl[h(t)\bigr]\Bigr]\in
{\rm C}^1\big([0,\infty), {\rm C}^2([0,1])\big).
$$
Consequently,
$$
(h,u)\in {\rm C}^1\big([0,\infty),(0,\infty)\times
C^2([0,1])\big)
$$
is the unique solution of~\eqref{BBB}. Moreover, $0\le u(t,y)\le c_*$
for $t\ge 0$ and $y\in [0,1]$. This completes the proof of Theorem
\ref{T1}.
\end{proof}

\subsection{Large-Time Behavior}

As already pointed out we can give a fairly complete picture of the dynamics for~\eqref{BBx}. We begin by investigating  equilibria.

\subsubsection{Equilibria}

Equilibria $(c,h)$ of \eqref{BBx} satisfy
\begin{subequations}\label{BBex}
\begin{align}
\kappa c_{zz}&= r(c),\quad 0<z<h,\label{BBex2}\\
c_z(0)&=0,\quad c(h)+\frac{L\kappa}{\kappa_L}c_z(h)=c_*, \label{BBex4}\\
&\int_0^h g\big(r(c(z))\big)\,\rd z=0. \label{BBex6}
\end{align}
 \end{subequations}
We establish the existence of equilibria for a general function
$g$. As it will be instrumental in the proof, we start with the
dependence of the solution $u=u[h]$ of \eqref{BBBB} as $h\to 0$ and as
$h\to \infty$.  We first deal with the limit as $h\to 0$.

\begin{lem}\label{h=0}
Assume~\eqref{GeneralAssumptions} with $r'\ge 0$. Given $h>0$, let $u[h]\in 
{\rm C}^2([0,1])$ be the unique solution to~\eqref{BBBB}
provided by Proposition~\ref{P1}. Then
$$
u[h]\to c_*\  \text{ in }\ {\rm C}^2([0,1])\ \text{ as }\ h\to 0.
$$
\end{lem}
\begin{proof}
We know that there is a unique solution $u[h]$ of
\eqref{BBBB} for every $h>0$ that satisfies
$u[h](y)\in[0,c_*]$ for $y\in [0,1]$. Integrating \eqref{BBBB1} and
using the boundary conditions \eqref{BBBB2} we obtain
\begin{align}\label{t1}
  h^2\int_0^y r\big(u[h](\eta)\big)\,\rd \eta=
  \kappa u_y[h](y),\quad y\in[0,1],
\end{align}
and thus, since $r\big(u[h](y)\big)\leq r(c_*)$, that
$$
0\leq \frac{u_y[h](\cdot)}{h}\leq\frac{r(c_*)}{\kappa}h\to 0
\ \text{ in }\ [0,1] \ \text{ as }\ h\to 0,
$$
so that \eqref{ccc2} implies that
$$
0\leq
c_*-u[h](1)=\frac{L\kappa}{\kappa_L}\frac{u_y[h](1)}{h}\to 0\
\text{ as }\ h\to 0.
$$
This entails that $u[0](1)=c_*$ and we deduce that $u[h]\to c_*$
in~${\rm C}^1([0,1])$, and hence in~${\rm C}^2([0,1])$ thanks to
\eqref{BBBB1}.
\end{proof}
Next we consider the limit as $h\to\infty$ and introduce the new
parameter $\varepsilon=h^{-2}$ and the new independent variable
$$
x=\frac{1-y}{\sqrt{\varepsilon}},\qquad
v(x)=u\left(1-\sqrt{\varepsilon}x\right) 
$$
so that \eqref{BBBB}
becomes
\begin{align*}
  &\kappa  v_{xx}=r(v), \quad
    x\in\bigl(0,1/\sqrt{\varepsilon}\bigr),\\ 
  &v_x(1/\sqrt{\varepsilon})=0,\qquad
  v(0)-\frac{L\kappa}{\kappa_L}\varepsilon v_x(0)=c_*.
\end{align*}
Clearly, it still holds that $v\in[0,c_*]$.

\begin{lem}\label{h=infty}
Assume~\eqref{GeneralAssumptions} with $r'(s)>0$ for $s\in [0,c_*]$.
Then  $u[h]\to 0$  uniformly  
in any compact subinterval of $[0,1)$ as $h\to\infty$. 
\end{lem}

\begin{proof}
First notice from~\eqref{ccc1} that  $0<v(x)\le v(0)<c_*$ for every
$x\in \left[ 0,1/\sqrt{\varepsilon}\right]$.  By assumption there is
$\alpha>0$ such that $r'(s)\ge \alpha$ for $s\in [0,c_*]$. 
Setting $w=v_x$ and noticing from~\eqref{ccc2} that $w\le 0$  in
$\left[ 0,1/\sqrt{\varepsilon}\right]$ we obtain by differentiating
the equation for $v$ that 
\begin{align*}
  &\kappa w_{xx}=r'(v)w\ge \alpha w, \quad x\in \bigl(
    0,1/\sqrt{\varepsilon}\bigr), \\
  & w(1/\sqrt{\varepsilon})=0,\quad
    w(0)=-\frac{c_*-v(0)}{L\kappa}\kappa_L.
\end{align*}
Introducing the differential operator $\mathcal{L}:=-\kappa\partial_x^2+\alpha$ so that $\mathcal{L}w\le 0$ in $\left[0,1/\sqrt{\varepsilon}\right]$, it follows from the comparison principle that $w\le \overline{w}$ in $\left[0,1/\sqrt{\varepsilon}\right]$, where $\overline{w}$ solves
$$
\mathcal{L}\overline{w}=0,\quad  x\in (0,1/\sqrt{\varepsilon}),
\qquad\overline{w}(1/\sqrt{\varepsilon})=0,\quad
\overline{w}(0)=-\frac{c_*-v(0)}{L\kappa}\kappa_L,
$$
that is,
$$
\overline{w}(x)= A_\alpha e^{-x\sqrt{\alpha}}+B_\alpha
e^{x\sqrt{\alpha}}, \quad x\in \left[ 0,1/\sqrt{\varepsilon}\right],
$$
with
$$
  A_\alpha=\frac{v(0)-c_*}{\beta}\frac{1}{1-e^{-2\sqrt{\alpha/\varepsilon}}}, \qquad
  B_\alpha=\frac{c_*-v(0)}{\beta}\frac{e^{-2\sqrt{\alpha/\varepsilon}}}{1-e^{-2\sqrt{\alpha/\varepsilon}}}.
$$
In the old variables we see that
\begin{align*}
0\le \sqrt{\varepsilon}u_y(y)\leq
\frac{c_*-v(0)}{\beta}\frac{1}{1-e^{-2\sqrt{\alpha/\varepsilon}}} 
\Big\{e^{-(1-y)\sqrt{\alpha/\varepsilon}}+e^{-(1+y)\sqrt{\alpha/\varepsilon}}\Bigr\}
\end{align*}
for $y\in [0,1]$. This shows, that $u_y[h](y) \to 0$
uniformly with respect to $y\in [0,1-\delta]$ as $h\to \infty$ for
each $\delta\in (0,1)$. Since 
$$
u_y[h](y)\ge\frac{h^2}{\kappa}y r\big(u[h](0)\big), \quad y\in[0,1],
$$
thanks to \eqref{BBBB1} and the monotonicity of $r$, it thus follows
that $r(u[h](0))\to 0$ as $h\to \infty$, hence $u[h](0)\to 0$ as $h\to
\infty$. Consequently, it follows from
$$
0< u[h](0)\le u[h](y)=u[h](0)+\int_0^y u_y[h](\eta)\,\rd \eta
$$
that $u[h]\to 0$  uniformly  
in any compact subinterval of $[0,1)$ as $h\to\infty$.
\end{proof}
\begin{rem}
The result contained in the above lemma describes a behavior that is
related to what is known in the applied literature as the existence of
micro-environments, in this case a sharp transition between regions of
high and low bio-activity. See \cite[Section 3.2]{KD10}.
\end{rem}
\begin{prop}
Assume~\eqref{GeneralAssumptions} with $r'(s)>0$ for $s\in [0,c_*]$. Let $g\in
{\rm C}^{1-}\bigl([0,\infty)\bigr)$ satisfy $g(0)<0$ and
$g\bigl( r(c_*)\bigr)>0$. Then system~\eqref{BBex} has
a solution, i.e. there is an equilibrium solution for the original
quasi-stationary system~\eqref{BBx}.
\end{prop}

\begin{proof}
Since $g\bigl( r(c_*)\bigr)>0$ there exists $a\in (0,c_*)$ such that
$g(r(s))>0$ for $s\in [a,c_*]$. From Lemma \ref{h=0} we infer that
there exists $h_0>0$ such that $ u[h](y)\in [a,c_*]$ for each $y\in
[0,1]$ and $h\in (0,h_0)$. Therefore, 
\begin{align*}
\int_0^1 g\big(r(u[h](y)\big)\big)\,\rd y>0, \quad h\in (0,h_0).
\end{align*} 
On the other hand, since $g\big(r(0)\big)=g(0)<0$, there is $b>0$ such
that $g\bigl( r(s)\bigr )\le g(0)/2<0$ for $s\in [0,b]$. Fix $\delta\in (0,1)$
with
$$
(1-\delta)\frac{g(0)}{2}+\delta\|
g\|_{\infty,[0,r(c_*)]}  <0. 
$$ 
Lemma \ref{h=infty} implies that there exists  $h_1>0$ such that
$u[h](y)\le b$ for each $y\in [0,1-\delta]$ and $h\ge
h_1$. Consequently,
\begin{align*} 
\int_0^1 g\bigl(r(u[h](y)\big)\bigr)\,\rd
  y&=\left(\int_0^{1-\delta}+\int_{1-\delta}^1\right)
     g\bigl( r\bigl( u[h](y)\bigr)\bigr)\,\rd y\\
&\le(1-\delta)\frac{g(0)}{2}+\delta\|
g\|_{\infty,[0,r(c_*)]}<0
\end{align*}
for $h\ge h_1$.
Since $u[h]$ depends continuously on $h$ by Proposition~\ref{P1},
the intermediate value theorem now yields the existence of a zero of
$\int_0^1g \bigl( r \bigl( u[h](y)\bigr)\bigr)\, dy$, i.e. of an equilibrium.
\end{proof}
It turns out that more can be said in the particular case when $g$ has
the special form
\begin{align}\label{gg}
g(s):=\alpha(s-b),\quad s\in [0,r(c_*)],
\end{align}
for some $\alpha,b>0$. Note that, for any solution $(c,h)$ of
\eqref{BBex}, assumption~\eqref{gg} implies that 
$$
\int_0^h g\big(r(c(z))\big)\,\rd z=\alpha\int_0^h r(c(z))\,\rd z
-\alpha bh=\alpha \big(\kappa c_z(h)- bh\big).
$$
Therefore, \eqref{BBex} with $g$ given by \eqref{gg} is equivalent to
finding $h>0$ and $c\in {\rm C}^2\bigl([0,h]\bigr)$ satisfying
\begin{subequations}
\label{BBeex}
\begin{align}
\kappa c_{zz}(z)&= r\bigl( c(z)\bigr),\quad 0<z<h,\label{BBeex2}\\
c_z(0)=0,\quad c(h)&=c_*-\frac{L b}{\kappa_L}h,\quad
c_z(h)=\frac{b}{\kappa}h.\label{BBeex4} 
\end{align}
 \end{subequations}
While existence of an equilibrium has already been established, a
shooting argument relying on the following lemma and the special
choice of $g$ also yield uniqueness of the equilibrium and its stability.
\begin{lem}\label{L10}
Let $r\in C^1(\mathbb{R})$ be bounded, strictly increasing,
and satisfy $s r(s)\ge 0$ for $s\in \mathbb{R}$. Given any
$c_0\ge 0$ there is a unique global solution $c(\cdot,c_0)\in
{\rm C}^2 \bigl( [0,\infty)\bigr) $ to the initial value
problem
\begin{subequations}
\label{c1}
\begin{align}
\kappa\, c_{zz}&= r( c ),\quad z>0,\label{c2}\\
c(0)&=c_0,\quad  c_z(0)=0. \label{c3} 
\end{align}
 \end{subequations}
The mappings $c(\cdot,c_0)$ and $c_z(\cdot,c_0)$ are increasing and
$c(\cdot,0)\equiv 0$. In addition,
$$
\bigl[ c_0\mapsto c(\cdot,c_0)\bigr]
\in{\rm C}^1\bigl([0,\infty),{\rm C}^2([0,R])\bigr)$$
for each $R>0$. Moreover, $w:=\frac{\partial}{\partial c_0} c(\cdot,c_0)\in
{\rm C}^2\bigl([0,\infty)\bigr)$ satisfies
\begin{subequations}
\label{w1}
\begin{align}
\kappa\, w_{zz}&= r'\bigl( c(\cdot,c_0)\bigr) w(z),\quad
z>0,\label{w2}\\  w(0)&=1,\quad w_z(0)=0, \label{w3} 
\end{align}
\end{subequations}
and it holds that $w(z)\ge 1$ as well as $w_z(z)\ge 0$ for $z\ge0$.
The function $M$, defined by
$$
M(z,c_0):=\left[\frac{Lb}{\kappa_L}+c_z(z,c_0)\right]
w_z(z)-\frac{r\bigl( c(z,c_0)\bigr)-b }{\kappa}w(z),\quad z\ge 0,
$$
is strictly increasing with respect to $z\ge 0$. In particular, for $r(c_0)\le b$ it holds
that $M(z,c_0)> M(0,c_0)\ge 0$ for $z> 0$. 
\end{lem}

\begin{proof}
The boundedness of $r\in C^1(\mathbb{R})$ ensures that, for each
$c_0\ge 0$, there is a unique global  solution $c(\cdot,c_0)\in
{\rm C}^2 \bigl( [0,\infty)\bigr)$ of the initial value
problem \eqref{c1} which satisfies
$$
c(z,c_0)=c_0+\frac{1}{\kappa}\int_0^z(z-\sigma) r\bigl(
c(\sigma,c_0)\bigr)\,\rd\sigma,\quad z\ge 0.
$$
Note that $c(\cdot,0)\equiv 0$ since $r(0)=0$. Similarly as in the
proof of Proposition \ref{P1}, the implicit function  theorem implies
that
$$
\bigl[ c_0\mapsto c(\cdot,c_0)\bigr] \in {\rm C}^1\bigl(
[0,\infty), {\rm C}^2\bigl([0,R]\bigr)\bigr)
$$
for each $R>0$. In order to include $c_0=0$ notice that the initial
value problem \eqref{c1} can be solved also for $c_0<0$. Let
$c_0\ge 0$ be fixed. Since $s\, r(s)\ge 0$ for $s\in \R$ and
$c(0,c_0)=c_0\ge 0$, it readily follows from \eqref{c1} that
$c(\cdot,c_0)$ is convex and hence that the maps
$c(\cdot,c_0)$ and $c_z(\cdot,c_0)$ are increasing. In particular,
$c(z,c_0)\geq 0$ for $z\ge 0$. \\[.1cm]
Differentiating \eqref{c1} with respect to $c_0$ we see that
$w:=\frac{\partial}{\partial c_0}c(\cdot,c_0)\in{\rm C}^2\bigl([0,\infty)\bigr)$
satisfies~\eqref{w1}. Hence, since $r'\bigl( c(z,c_0)\bigr)\ge 0$ as
$c(z,c_0)\geq c_0\geq 0$ for $z\ge 0$ and $w(0)=1$, we conclude that
$w$ is (at least initially) convex. It follows that $w_z(z)\geq 0$ and
\mbox{$w(z)\ge w(0)=1$} for $z\geq 0$ small and then for all $z\geq 0$ by
monotonicity. Hence~\eqref{c2} and~\eqref{w2} imply  that
\begin{equation*}
\frac{d}{dz}M(z,c_0)=\frac{Lb}{\kappa_L\kappa}r'\bigl( c(z,c_0)\bigr)
w(z)+\frac{b}{\kappa}w_z(z)> 0, \quad z> 0,
\end{equation*}
which makes $M$ strictly increasing. It remains to notice that, for $r(c_0)\leq
b$,
$$
M(0,c_0)=-\frac{1}{\kappa}\left[ r(c_0)-b \right]\geq 0
$$
due to \eqref{c3} and \eqref{w3}.
\end{proof}

\begin{prop}\label{P2}
Assume \eqref{GeneralAssumptions} with $r'(s)>0$ for $s\in [0,c_*]$. It holds:
\begin{itemize}
\item[(i)] If $g\in {\rm C}^{1-}\bigl( [0,\infty)\bigr)$ is
  strictly increasing and $g \bigl( r(c_*)\bigr)\leq 0$, then the equilibrium
  problem~\eqref{BBex}  has no solution.
\item[(ii)] If $g$ is given by \eqref{gg} and $r(c_*)> b$, then the
  equilibrium problem \eqref{BBex} (or, equivalently, \eqref{BBeex}) has a unique solution $(h_e,c_e)$ 
  with $h_e>0$  and $c_e\in  {\rm C}^2([0,h_e])$.
\end{itemize}
\end{prop}
\begin{proof}
(i) If $(c,h)$ is a solution of \eqref{BBex}, then $c$ is
  non-decreasing and $c(z)< c_*$ for $z\in [0,h)$. It follows that
  $g\bigl( r\bigl(c(z)\bigr)\bigr)< 0$ for $z\in [0,h)$ thanks to
  Proposition \ref{P1}, so that
  $\int_0^h g\bigl( r\bigl(c(z)\bigr)\bigr)\,\rd z<0$ and,
  consequently, that there is no equilibrium in this case.

(ii) In order to establish that there exists a unique
  equilibrium we use a shooting argument. To this end, let
  $c(\cdot,c_0)\in {\rm C}^2([0,\infty))$ be the unique solution of
  the initial value problem~\eqref{c1} provided by Lemma \ref{L10} for
  $c_0\ge 0$. Here we may assume that $r$ satisfies the assumptions
  of Lemma~\ref{L10} as only $r\big |_{[0,c_*]}$ plays a role,
  see~\eqref{ccc1}.
  To enforce the boundary conditions \eqref{BBeex4} we define the
  functions $A$ and $B$ by
  $$
  A(z,c_0):=c(z,c_0)+\frac{L b}{\kappa_L}z-c_*, \qquad
  B(z,c_0):=c_z(z,c_0)-\frac{b}{\kappa}z,
  $$
  for $c_0,z\ge 0$. Consider $c_0\in [0,c_*)$. Using Lemma
  \ref{L10} it follows that $A(0,c_0)=c_0-c_*< 0$ and that
  $A(\cdot,c_0)$ is increasing and 
  unbounded above. Thus, there is a unique $h(c_0)> 0$ such that
  $A\bigl(h(c_0), c_0\bigr)=0$. Since 
$$
\partial_z A(h(c_0), c_0)= c_z(h(c_0),c_0)+\frac{L b}{\kappa_L}\ge
\frac{L b}{\kappa_L}>0, 
$$ 
it follows from the implicit function theorem that
$h\in{\rm C}^1\bigl([0,c_*)\bigr)$ with $h(0)=\kappa_L
c_*/(Lb)$.
Then 
  $$
  B\bigl( h(\cdot),\cdot\bigr)\in{\rm C}^1\bigl([0,c_*)\bigr).
  $$
  Since $r$ is strictly increasing and $r(0)=0$, there is a unique
  $\underline{c}\in (0,c_*)$ such that $r(\underline{c})=b$. Thus, for
  $c_0\in (\underline{c},c_*)$ we have that
  $$
  c(z,c_0)\geq c(0,c_0)=c_0>\underline{c},\quad z\geq 0,
  $$
  and hence $r\bigl( c(z,c_0)\bigr)>b$ for $z>0$ since
  $r$ is increasing. This and \eqref{c2} yield that
  $$
  \partial_z
  B(z,c_0)=c_{zz}(z,c_0)-\frac{b}{\kappa}=\frac{1}{\kappa}
  \big[r\bigl(c(z,c_0)\bigr)-b\big]>0,\quad z>0,
  $$
  while $B(0,c_0)=c_z(0,c_0)=0$. Consequently, we have that $B\bigl(
  h(c_0),c_0\bigr)>0$ for $c_0\in (\underline{c},c_*)$. On the other
  hand we have that
  $$
  B(h(0),0)=-\frac{\kappa_L c_*}{L\kappa}<0,
  $$
  and, since $B\bigl( h(\cdot),\cdot\bigr)$ depends continuously on
  its argument $c_0$, there must exist some ${\sf c}_e\in (0,\underline{c})$
  such that $B\bigl (h({\sf c}_e),{\sf c}_e\bigr)=0$. By construction, we have
  that $c(\cdot,{\sf c}_e)\in{\rm C}^2\bigl([0,h_{{\sf c}_e}]\bigr)$
  and that $(h({\sf c}_e),c\bigl( \cdot,{\sf c}_e)\bigr)$
  solves \eqref{BBeex} (where $h({\sf c}_e)>0$). In fact, we claim that
  $B\bigl( h(\cdot),\cdot\bigr)$ is strictly increasing
  on~$[0,\underline{c}]$ so that it can only possess a single zero and 
  the equilibrium is thus unique. To see this, we first differentiate
  the identity $A(h(c_0),c_0)=0$ with respect to $c_0$ and use
  Lemma~\ref{L10} to obtain 
\begin{align}\label{i1}
\left(\frac{Lb}{\kappa_L}+ c_z(h(c_0),c_0)\right)\frac{d}{d c_0}h(c_0)=-w(h(c_0))\le -1,
\end{align}
where $w=\frac{\partial}{\partial{c_0}} c(\cdot,c_0)$ as in Lemma \eqref{L10}. Next,
  we observe that
  \begin{equation*}
    \frac{d}{d c_0}  B\bigl( h(c_0), c_0\bigr)=w_z\bigl(
    h(c_0)\bigr)+\left[ c_{zz}\bigl(h(c_0),c_0\bigr)-\frac{b}{\kappa}\right]
    \frac{d}{dc_0}h(c_0),
  \end{equation*}
  and hence, using \eqref{c2} and \eqref{i1}, that
  \begin{align*}
    &\left(\frac{Lb}{\kappa_L}+c_z\bigl( h(c_0),c_0\bigr)\right)
    \frac{d}{dc_0}  B\bigl( h(c_0),c_0\bigr) \\
    &\quad =\left(\frac{Lb}{\kappa_L}+c_z\bigl(
      h(c_0),c_0\bigr)\right)
    w_z(h(c_0))-\frac{1}{\kappa}\bigl[r\bigl(c(h(c_0),c_0)\bigr)-b\bigr]
    w\bigl( h(c_0)\bigr)\\
    &\quad=M\bigl( h(c_0),c_0\bigr)>0,
  \end{align*}
  where we used that $M$, which was defined in Lemma \ref{L10}, is
  strictly positive since $r(c_0)\le b$ for $c_0\in[0,\underline{c}]$. Since Lemma
  \ref{L10}
  also ensures that $\frac{Lb}{\kappa_L}+c_z\bigl(
  h(c_0),c_0\bigr)>0$, we indeed have that $B\bigl(
  h(\cdot),\cdot\bigr)$ is strictly increasing on
  $[0,\underline{c}]$. Thus there is a single solution of
  \eqref{BBeex}, as claimed.
\end{proof}

The next result is the analogue of Proposition~\ref{P5}.

\begin{prop}\label{P3}
Assume~\eqref{GeneralAssumptions} and that $r\big |_{[0,c_*]}$ and
$g\in \operatorname{C}^{1-}\Big([0,r(c_*)]\Bigr)$ are increasing with
$g\bigl(r(c_*)\bigr)\leq 0$. Given $h_0>0$ let
$$
(h,u)\in {\rm C}^1\big( \mathbb{R}^+,(0,\infty)\times
{\rm C}^2([0,1]) \big)
$$
be the unique solution to \eqref{BBB} obtained in Theorem
\ref{T1}. Then
$$
h(t)\searrow 0\ \text{ as }\ t\to\infty,
$$
and
$$
u(t,\cdot)\to c_*\  \text{ in } \ {\rm C}^2([0,1])\ \text{ as }\ t\to\infty.
$$
Furthermore,
\begin{align}\label{t0}
  \lim_{t\to\infty}\frac{c_*-u(t,1)}{h(t)}= \frac{L
  r(c_*)}{\kappa}\,,\qquad
  \lim_{t\to\infty}\frac{u_y(t,1)}{h(t)^2}= \frac{r(c_*)}{\kappa}.
\end{align}
\end{prop}

\begin{proof}
Since $0\le u(t,y)\le c_*$ for $t\geq 0$ and $y\in [0,1]$ thanks to
\eqref{ccc1}, we have that $g\bigl(r\bigl(u(t,y)\bigr)\bigr)\leq 0$
and thus $h_t\leq 0$. Since there is no equilibrium according to Proposition~\ref{P2}, it follows
that $h(t)\searrow 0$ as $t\to \infty$. In fact, since
$h(\mathbb{R}^+)$ is bounded, we infer from \eqref{ccc} and the
Arzel\`{a}-Ascoli theorem that $\big\{u(t,\cdot)\, \big |\, t\geq 0\}$ is
relatively compact in ${\rm C}^1([0,1])$. Thus, given
$t_j\nearrow\infty$, we may assume without loss of generality that
there is $u^*\in {\rm C}^1([0,1])$ such that $u(t_j,\cdot)\to
u^*$ in ${\rm C}^1([0,1])$. Since $h(t_j)\searrow 0$, it
follows from \eqref{ccc2} that $u_y^*(y)=0$ for $y\in [0,1]$ and it
must hold that $u^*=const$. Using arguments completely analogous to
those in the proof of Lemma \ref{h=0} we see that
$$
  c_*-u(t,1)=\frac{L\kappa}{\kappa_L}\frac{u_y(t,1)}{h(t)}\to 0\ \text{
  as }\ t\to \infty.
$$
This implies that $u^*=c_*$ and we deduce that $u(t,\cdot)\to c_*$ in
${\rm C}^1([0,1])$, and hence in~${\rm C}^2([0,1])$
thanks to \eqref{BBB2}. Finally, \eqref{t1} yields the second limit in
\eqref{t0}; the latter together with \eqref{ccc2} then implies also
the first limit in \eqref{t0} and the proof is complete.
\end{proof}

Notice that Proposition \ref{P3} covers the case when $r(c_*)\leq b$
with $g$ given by \eqref{gg}. We consider now the case when $r(c_*)> b$.

\begin{prop}\label{P4}
Assume~\eqref{GeneralAssumptions} with  $r'(s)>0$ for $s\in [0,c_*]$. Let $g$ be given by~\eqref{gg}
with $r(c_*)> b$. Denote by $(h_e,u_e)$ the unique equilibrium of
\eqref{BBB} with $h_e>0$ and $u_e\in{\rm C}^2([0,1])$ 
obtained in Proposition \ref{P2}. Given $h_0>0$, let 
$$
(h,u)\in {\rm C}^1\Big(\mathbb{R}^+,(0,\infty)\times
{\rm C}^2([0,1])\Bigr)
$$
be the unique solution of \eqref{BBB} provided by Theorem
\ref{T1}. Then
$$
h(t)\to h_e\ \text{ as }\ t\to\infty
$$
and
$$
u(t,\cdot)\to u_e \ \text{ in }{\rm C}^2([0,1])\ \text{ as }\ t\to\infty.
$$
\end{prop}

\begin{proof}
 Since $h_e$ is the only zero of the function $f:(0,\infty)\to
 \mathbb{R}$ defined in \eqref{f} thanks to Proposition \ref{P2} and
 given that $h_t=f(h)$, it follows that either $h_t<0$ or
 $h_t>0$ on $[0,\infty)$. In fact, Assumption \eqref{gg} and
 \eqref{BBBB} imply that
\begin{align*}
f\bigl(h(t)\bigr)&=\alpha h(t)\left(\int_0^1r\big(u(t,y)\big)\,\rd
y-b\right)=\alpha h(t)\left(\frac{\kappa}{h(t)^2}u_y(t,1)-b\right)\\
&=\frac{\alpha\kappa_L}{L}\Big(c_*-u(t,1)\Big)-\alpha bh(t)\leq
\frac{\alpha\kappa_L}{L} c_* -\alpha bh(t)
\end{align*}
so that there is $m(h_0)>0$ such that $0\leq h(t)\le m(h_0)$ for $t\geq
0$. The monotonicity of $h$ thus ensures that 
$$
h_*:=\lim_{t\to\infty} h(t) \in \bigl[0, m(h_0)\bigr]
$$
exists. We claim that $h_*>0$. Since $r\bigl(u(t,y)\bigr)\leq r(c_*)$
for $t\ge 0$ and $y\in [0,1]$, it follows from \eqref{C} and
\eqref{ccc1} that, for $t\geq 0$ and $y\in [0,1]$, it must hold that
\begin{align*}
u(t,y)\ge u(t,0)&=c_*-\frac{Lh(t)}{\kappa_L}\int_0^1 r\bigl(
u(t,y)\bigr)\,\rd y 
-\frac{h(t)^2}{\kappa} \int_0^1 (1-y)\, r\bigl(u(t,y)\bigr)\,\rd y\\
&\geq c_*- \frac{Lh(t)}{\kappa_L}r(c_*)- \frac{h(t)^2}{\kappa}r(c_*).
\end{align*}
Let $\underline{c}\in (0,c_*)$ be such that
$r(\underline{c})=b$. Then, there are $\ve>0$ and
$\delta(\underline{c})>0$ such that 
$$
c_*- \frac{L z}{\kappa_L}r(c_*)-\frac{z^2}{\kappa}r(c_*)>
\underline{c}+\ve,\quad 0\leq z\leq \delta(\underline{c}).
$$
In other words, whenever $0\leq h(t)\le \delta(\underline{c})$, we
have that $r\bigl(u(t,y)\bigr)>r(\underline{c}+\ve)$ for each $y\in
[0,1]$ and thus
$$
h_t(t)\geq \alpha  h\big[r(\underline{c}+\ve)-b\big]>0.
$$
There is therefore a $t_0>0$ such that $h(t)\geq\delta(\underline{c})$
for each $t\geq t_0$, hence indeed $h_*>0$. But then $h_*=h_e$ is
necessarily the unique zero of $f$. As in the proof of Proposition~\ref{P3} we infer from \eqref{ccc} and the Arzel\`a-Ascoli theorem
that $\big\{u(t,\cdot)\, \big |\, t\geq 0\big\}$ is relatively compact
in ${\rm C}^1([0,1])$. Thus, given $t_j\nearrow\infty$ as
$j\to\infty$ we may assume, without loss of generality, that there is $u^*\in
{\rm C}^1([0,1])$ such that $u(t_j,\cdot)\to u^*$ in
${\rm C}^1([0,1])$ as $j\to\infty$. But since $h(t_j)\to
h_e$, as just shown above, it holds that $u^*=u_e$. Therefore,  we
conclude that $u(t,\cdot)\to u_e$ in~${\rm C}^1([0,1])$ and,
hence, in ${\rm C}^2([0,1])$ thanks to \eqref{BBB2}. The
proof is finished.
\end{proof}

\begin{rem}
An interesting open problem is whether the non-trivial steady-state $(h_e,c_e)$ of Proposition \ref{P2} is actually stable for the full
evolutionary problem \eqref{EBBx}. Numerical
simulations seem to indicate this. While the solution $u=c_*-v$ in
the transformed variables \eqref{trafo} appears to very quickly become
monotone increasing, the biofilm depth $h$ converges in an oscillatory
fashion on a somewhat larger time scale. A calculation shows that the
linearization for $h$ vanishes in the steady-state so that even local
asymptotic stability is not a consequence of the principle of
linearized stability. This is clearly visible in the numerical
experiment shown in Figure \ref{fig:F1}, where $r=2\tanh(\cdot)$, the initial
configuration is given by $h_0=3.5$ and $u_0=\cos^2(2\pi y)$, and the
parameters used are $\varepsilon=1$, $c_*=1$, $\alpha=1$, and $b=.5$.
\begin{figure}
\begin{center}
  \includegraphics[scale=.65]{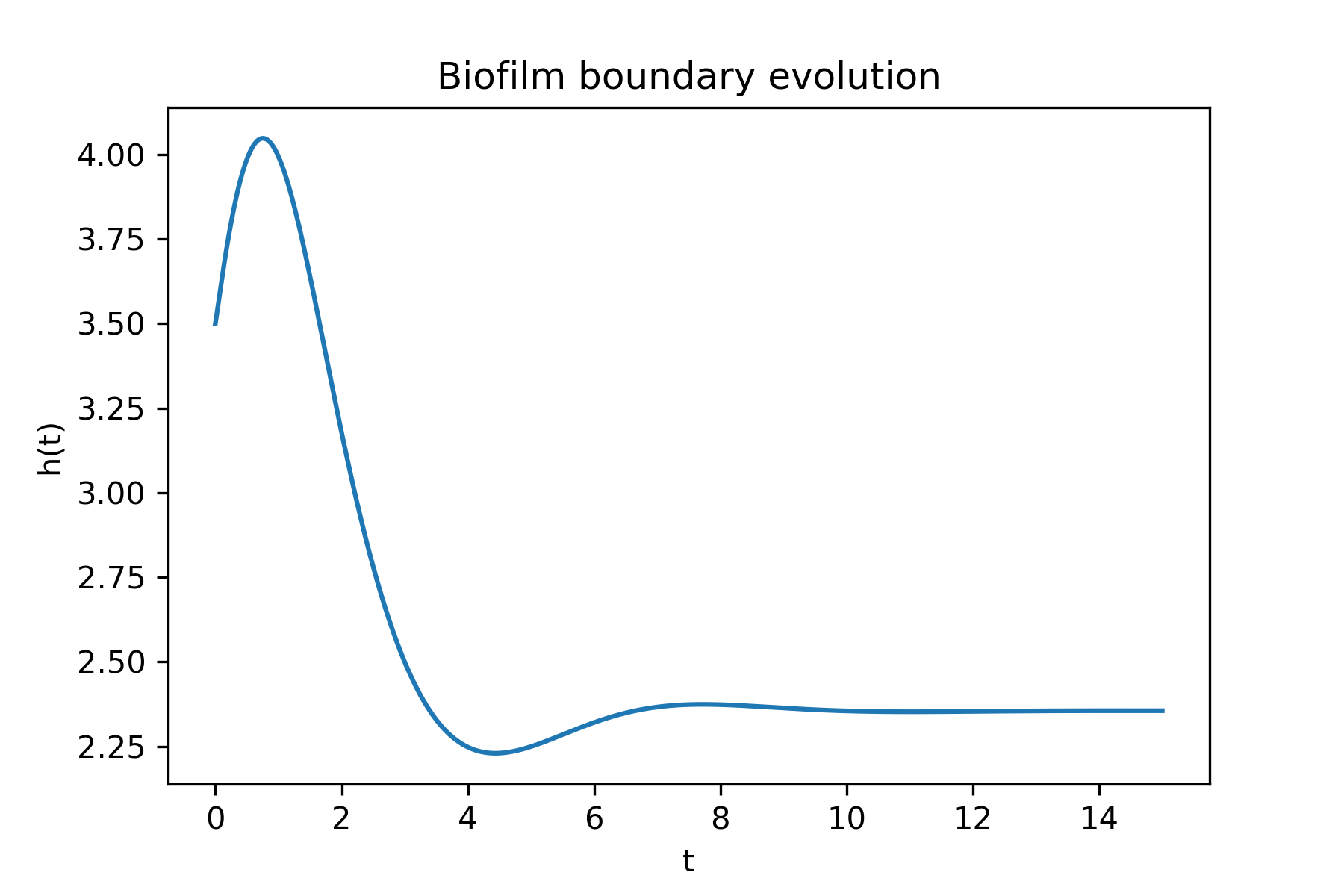}
  \includegraphics[scale=.4]{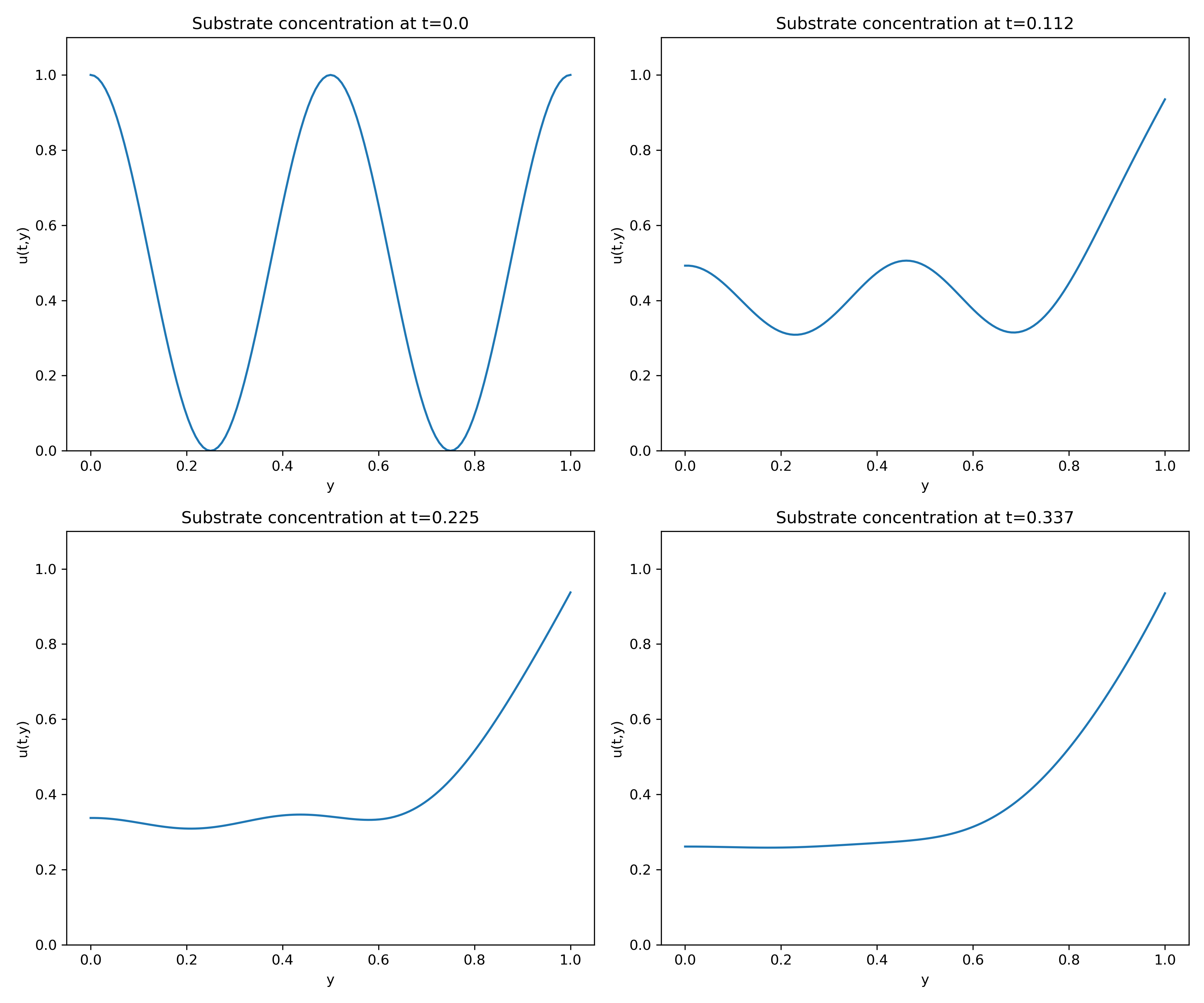}
  \caption{The "long" time evolution of the moving boundary (film depth) and that of substrate concentration at small times.}
  \label{fig:F1}
\end{center}
\end{figure}
The simulation highlights the fact that diffusion quickly removes
oscillations in the initial concentration (Figure \ref{fig:F1}) but, within the class of
monotone substrate concentrations, the convergence to equilibrium does
not appear to be exponential nor monotone (Figure \ref{fig:F2}).
\begin{figure}
  \begin{center}
    \includegraphics[scale=.65]{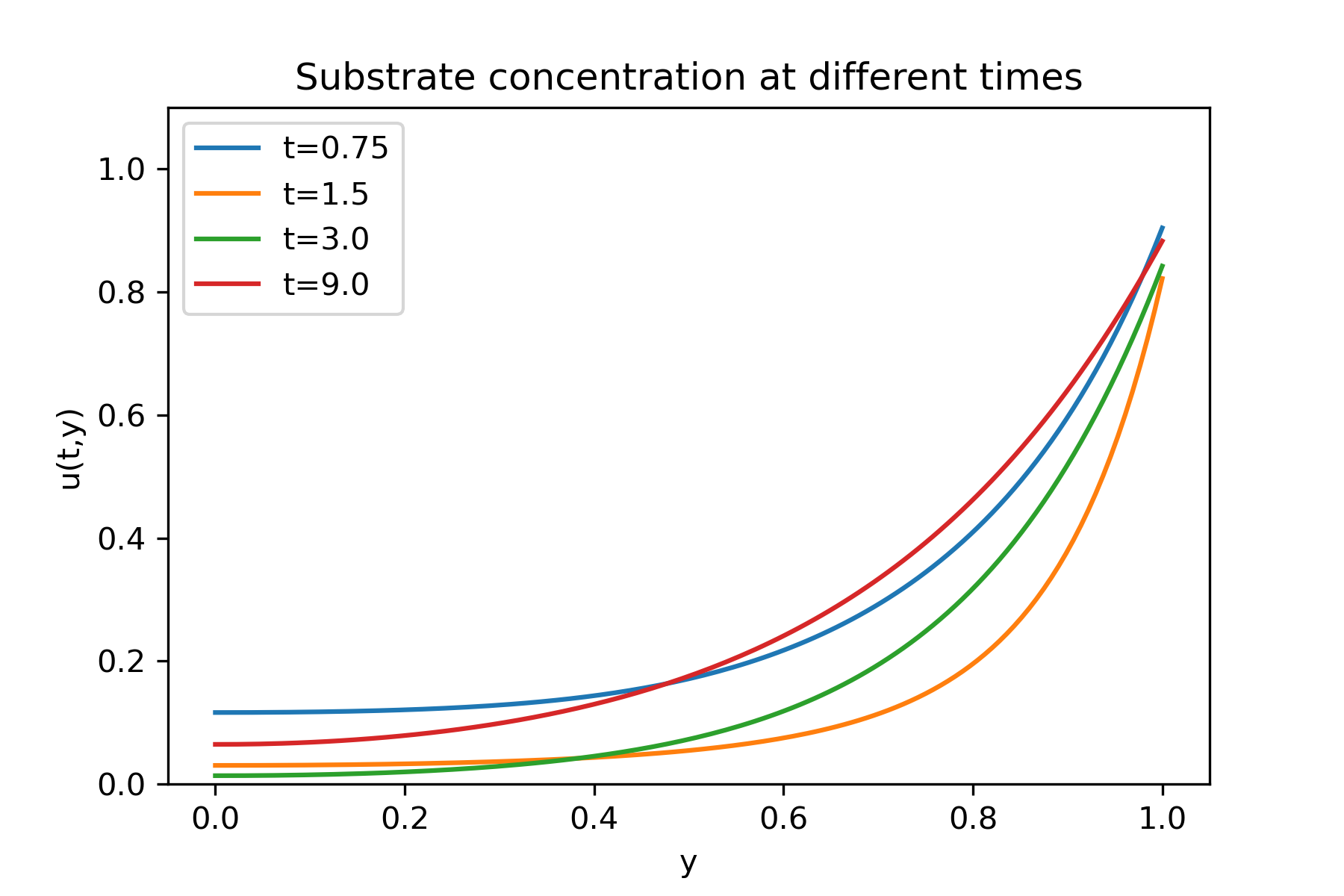}
    \caption{The behavior of substrate concentration at larger times.}
    \label{fig:F2}
  \end{center}
\end{figure}
\end{rem}
\bibliographystyle{siam}
\bibliography{BiofilmLiterature}

\end{document}